\documentclass[a4paper, reqno, 12pt]{amsart}
\usepackage{stmaryrd}
\usepackage{float}

\usepackage[usenames,dvipsnames]{color}
\usepackage{amsthm,amsfonts,amssymb,amsmath,amsxtra}
\usepackage{tikz}
\usepackage{tkz-euclide}
\usepackage[all]{xy}
\SelectTips{cm}{}
\usepackage{xr-hyper}
\usepackage[colorlinks=
   citecolor=Black,
   linkcolor=Red,
   urlcolor=Blue]{hyperref}
\usepackage{verbatim}

\usepackage[margin=1.25in]{geometry}
\usepackage{mathrsfs}

\RequirePackage{xspace}
\RequirePackage{etoolbox}
\RequirePackage{varwidth}
\RequirePackage{enumitem}
\RequirePackage{tensor}
\RequirePackage{mathtools}
\RequirePackage{longtable}
\RequirePackage{multirow}

\def\ge{\geqslant}
\def\le{\leqslant}
\def\a{\alpha}
\def\b{\beta}
\def\g{\gamma}
\def\G{\Gamma}

\def\o{\omega}

\def\s{\sigma}
\def\t{\tau}
\def\th{\theta}
\def\k{\kappa}
\def\l{\lambda}

\def\i{^{-1}}

\def\<{\langle}
\def\>{\rangle}

\newcommand{\sA}{\ensuremath{\mathscr{A}}\xspace}
\newcommand{\sB}{\ensuremath{\mathscr{B}}\xspace}

\newcommand{\sP}{\ensuremath{\mathscr{P}}\xspace}

\newcommand{\sU}{\ensuremath{\mathscr{U}}\xspace}

\newcommand{\fka}{\ensuremath{\mathfrak{a}}\xspace}

\newcommand{\bA}{\mathbf A}

\newcommand{\bG}{\mathbf G}

\newcommand{\bM}{\mathbf M}
\newcommand{\bN}{\mathbf N}

\newcommand{\bS}{\mathbf S}
\newcommand{\bT}{\mathbf T}

\def\br{\breve}

\def\brG{\breve G}
\def\brF{\breve F}
\def\brW{\breve W}
\def\brN{\breve N}
\def\brI{\breve \CI}

\def\brI{\breve I}

\def\bsA{\breve \sA}

\newcommand{{\BG}}{\ensuremath{\mathbb {G}}\xspace}

\newcommand{{\BK}}{\ensuremath{\mathbb {K}}\xspace}

\newcommand{\BQ}{\ensuremath{\mathbb {Q}}\xspace}
\newcommand{\BR}{\ensuremath{\mathbb {R}}\xspace}
\newcommand{\BS}{\ensuremath{\mathbb {S}}\xspace}

\newcommand{\BZ}{\ensuremath{\mathbb {Z}}\xspace}

\newcommand{\Ad}{{\mathrm{Ad}}}

\let\Im\relax
\DeclareMathOperator{\Im}{Im}

\newcommand{\Int}{\ensuremath{\mathrm{Int}}\xspace}

\DeclareMathOperator{\Res}{Res}

\newcommand{\SL}{{\mathrm{SL}}}

\newcommand{\lto}{\longmapsto}

\newtheorem{theorem}{Theorem}
\newtheorem{theoremA}{Theorem}

\newtheorem{proposition}[theorem]{Proposition}

\newtheorem{corollary}[theorem]{Corollary}

\theoremstyle{definition}

\newtheorem{example}[theorem]{Example}

\newtheorem{remark}[theorem]{Remark}

\numberwithin{equation}{section}
\numberwithin{theorem}{section}

\setitemize[0]{leftmargin=*,itemsep=\the\smallskipamount}
\setenumerate[0]{leftmargin=*,itemsep=\the\smallskipamount}

\renewcommand{\to}{%
   \ifbool{@display}{\longrightarrow}{\rightarrow}%
   }
\let\shortmapsto\mapsto
\renewcommand{\mapsto}{%
   \ifbool{@display}{\longmapsto}{\shortmapsto}%
   }
\newcommand{\isoarrow}{%
   \ifbool{@display}{\overset{\sim}{\longrightarrow}}{\xrightarrow\sim}%
   }

\begin{document}
\author[X.~He]{Xuhua He}
\address{Xuhua He, Department of Mathematics, University of Maryland, College Park, MD 20742, USA}
\email{xuhuahe@math.umd.edu}
\thanks{X.~H.~was partially supported by NSF DMS-1801352. S.N. is supported in part by QYZDB-SSW-SYS007 and NSFC grant (No. 11501547 and No. 11621061).}
\author[S.~Nie]{Sian Nie}
\address{Sian Nie, Institute of Mathematics, Academy of Mathematics and Systems Science, Chinese Academy of Sciences, 100190, Beijing, China}
\email{niesian@amss.ac.cn}

\title{A geometric interpretation of Newton strata}
\begin{abstract}
	The Newton strata of a reductive $p$-adic group  are introduced in \cite{Newton} and play some role in the representation theory of $p$-adic groups. In this paper, we give a geometric interpretation of the Newton strata.
\end{abstract}
\maketitle

\section*{Introduction}

Let $F$ be a $p$-adic field and $\bG$ be a connected reductive group over $F$. Let $\brF$ be the completion of the maximal unramified extension of $F$ and $\s$ be the Frobenius automorphism of $\brF$ over $F$. Let $G=\bG(F)$ and $\brG=\bG(\brF)$. Then $\s$ induces a Frobenius automorphism on $\brG$ which we still denote by $\s$. Then $G=\brG^\s$.

Let $B(\bG)$ be the set of $\s$-conjugacy classes of $\brG$. In case $F=\BQ_p$, $B(\bG)$ is the set of isomorphism classes of isocrystals with $\brG$-structure. It appears naturally in the study of Shimura varieties and the study of Rapoport-Zink space on period spaces for $p$-divisible groups. Recently, Fargues \cite{F1} gives a classification of $\brG$-torsors on the Fargues-Fontaine curve via $B(\bG)$, which is a starting point of his program to geometrize the local Langlands correspondence \cite{F2}.

The set $B(\bG)$ is classified by Kottwitz in \cite{kottwitz-isoI} and \cite{kottwitz-isoII} via the Kottwitz map and the Newton map. An analogous of the Newton map for the ordinary conjugation action on $\brG$, instead of the $\s$-conjugation action, is introduced by Kottwitz and Viehmann in \cite{KV}. This leads to a decomposition of $\brG$ into subsets stable under the conjugation action, with a discrete index set.

On the other hand, in \cite{Newton}, the first author introduces the Newton decomposition of the $p$-adic group $G$, a decomposition of $G$ into certain nice open and closed subsets, which are stable under the conjugation action. The strata in the decomposition are called the Newton strata. The Newton decomposition has found many applications to the study of representations of $p$-adic groups, e.g. a new proof of the Howe's conjecture on the invariant distributions, and the trace Paley-Wiener Theorem and abstract Selberg principle for mod-$l$ representations (see a series of papers \cite{Newton}, \cite{H2}, \cite{CH}). The precise definition of the Newton strata, however, is rather technical and involves the elements of the Iwahori-Weyl group $W$ of $G$ that are of minimal length in their conjugacy classes.

The main purpose of this paper is to give a geometric interpretation of the Newton strata of $G$, via the set $B(\bG)$ and its variations.

We first consider the case where $\bG$ is split over $F$. As pointed out in \cite[Introduction]{Newton}, the Kottwitz map and the Newton map of $G$ coincide with the restriction to $G$ of the Kottwitz map and the Newton map of $\brG$ respectively. In this case, the Newton stratum of $G$ is the intersection of $G$ with the corresponding $\s$-conjugacy class of $\brG$. A similar result holds if we replace the $\s$-conjugacy classes of $\brG$ by the subset of $\brG$ stable under the ordinary conjugation action with the same Kottwitz point and Newton point.

The situation is more subtle for non-split groups. One technical difficulty is that the Newton maps over $F$ and over $\brF$ do not coincide. The images of the Newton maps are rational coweights. The underline affine space, of the vector space spanned by coweights, may be identified with the apartment by mapping the zero coweight to a given special vertex in the apartment. However, in general, a special vertex in the apartment over $F$ may not be a special vertex, or even not a vertex in the apartment over $\brF$.

To overcome this difficulty, we use the straight conjugacy classes of the Iwahori-Weyl groups instead of the Kottwitz points and the Newton points. The straight conjugacy classes are certain nice conjugacy classes in the Iwahori-Weyl group. For the precise definition, we refer to \S \ref{straight}. Let $\brW \sslash_{\s} \brW$ be the set of straight $\s$-conjugacy classes of the Iwahori-Weyl group $\brW$ of $\brG$. In \cite{He14}, the first author establishes a natural bijection $$B(\bG) \cong \brW \sslash_{\s} \brW.$$ In other words, we have $$\brG=\sqcup_{\br \nu \in \brW \sslash_{\s} \brW} [\br \nu],$$ where $[\br \nu]$ is the $\s$-conjugacy class of $\brG$ corresponding to $\br \nu$.

Similarly, we consider the set of straight (ordinary) conjugacy classes $W \sslash W$ of the Iwahori-Weyl group of $G$. The Newton decomposition on $G$ is given by $$G=\sqcup_{\nu \in W \sslash W} G(\nu),$$ where $G(\nu)$ is the Newton stratum of $G$ corresponds to $\nu$. See \S\ref{5.1} for the precise definition.

We have the natural identification $W=\brW^\s$ which induces a natural map $$W\sslash W \to \brW \sslash_{\s} \brW.$$

This map, in some sense, provides a way to compare the Newton points over $F$ and over $\brF$. However, Kottwitz points over $F$ and over $\brF$ do not match in general. The solution to this problem is to replace $\brW \sslash_{\s} \brW$ by $\brW \sslash_{\s} \brW_a$, where $\brW_a$ is the affine Weyl group of $\brW$ and $\brW \sslash_{\s} \brW_a$ denotes the set of straight $\brW_a$-orbits on $\brW$ for the $\s$-conjugation action. We have

\begin{theoremA}[Theorem \ref{inject}]
	The natural embedding $W \to \brW$ induces an injection $$i: W\sslash W \hookrightarrow \brW \sslash_{\s} \brW_a.$$
\end{theoremA}

Now let us come back to the group $G$. Let $\brG_1$ be the subgroup of $\brG$ generated by the parahoric subgroups. Let $B(\bG)^\flat$ be the set of $\brG_1$-orbits on $\brG$ for the $\s$-conjugation action. Then we have $$B(\brG)^\flat \cong \brW \sslash_{\s} \brW_a.$$ In other words, we have $$\brG=\sqcup_{\br \nu \in \brW \sslash_{\s} \brW_a} [\br \nu]^\flat,$$ where $[\br \nu]^\flat$ is the $\brG_1$-orbit on $\brG$ corresponding to $\br \nu$ for the $\s$-conjugation action.

Now we have the following geometric interpretation of the Newton strata.

\begin{theoremA}[Theorem \ref{5-5}]
	Let $\nu \in W \sslash W$. Then we have $$G(\nu)=G \cap [i(\nu)]^\flat.$$
\end{theoremA}

In the main text of the paper, we also consider the Newton strata of $G$ for the twisted conjugation action and we obtain a similar geometric interpretation.

\smallskip
\smallskip

\noindent \ \ {\bf Acknowledgements:} We thank T. Haines, G. Lusztig and M. Rapoport for helpful discussions. 

\section{The $p$-adic groups and their Iwahori-Weyl groups}\label{setup}

\subsection{The Iwahori-Weyl group over $F$}
Fix a maximal $F$-split torus $\bA$ of $\bG$. Let $\bM$ and $\bN$ be the centralizer and normalizer of $\bA$ respectively. Let $$W_0=\bN(F) / \bM(F), \qquad W=\bN(F) / \bM(F)_1$$ be its relative Weyl group and Iwahori-Weyl group respectively, where $\bM(F)_1$ is the unique parahoric subgroup of $M(F)$. Let $\sB = \sB(\bG, F)$ be the enlarged Bruhat-Tits building and let $\sA=\sA(\bG, \bA, F)$ be the apartment of $\sB$ corresponding to $\bA$. Fix an alcove $\fka_C \subseteq \sA$, and denote by $I \subseteq G$ the associated Iwahori subgroup. Then we have $I \backslash G / I = W$. Let $G_1$ be the subgroup of $G$ generated by all the parahoric subgroups. The affine Weyl group of $G$ is given by $W_a=N_1 / (N_1 \cap I)$ with $N_1\ = G_1 \cap \bN(F)$. We have $$W = W_a \rtimes \Omega,$$ where $\Omega$ is the stabilizer of $\fka_C$. We denote by $\BS$ the set of simple reflections of $W_a$ and by $\ell$ the length function on $W$.

\subsection{The Iwahori-Weyl group over $\brF$} Let $\bS$ be a maximal $\brF$-split torus which is defined over $F$ and contains $\bA$. Let $\bsA$ be the apartment corresponding to $\bS$ over $\brF$. The Frobenius automorphism $\s$ of $\brF / F$ acts on $\bsA$, and there is a natural embedding $\sA \hookrightarrow \bsA$, which identifies $\sA$ with $\bsA^\s$. Let $\bT$ be the centralizer of $\bS$, and let $\bN_{\bS}$ be the normalizer of $\bS$. Let $$\brW_0=\bN_{\bS}(\brF) / \bT(\brF), \qquad \brW=\bN_{\bS}(\brF) / \bT(\brF)_1$$ be its Weyl group and Iwahori-Weyl group respectively, where $\bT(\brF)_1$ is the unique parahoric subgroup of $\bT(\brF)$. Fix an $\s$-stable alcove $\breve \fka_C \subseteq \bsA$ containing $\fka_C$, and denote by $\brI \subseteq \brG$ the associated Iwahori subgroup. Then $I=\brI^\s$ and $\brI \backslash \brG / \brI = \brW$. Let $\brG_1$ be the subgroup of $\brG$ generated by all parahoric subgroups. The affine Weyl group of $\brG$ is given by $\brW_a= \brN_1 / (\brN_1 \cap \brI)$ with $\brN_1 = \brG_1 \cap \bN_{\bS}(F)$. We have $$\brW = \brW_a \rtimes \breve\Omega,$$ where $\breve\Omega$ is the stabilizer of $\breve\fka_C$. We denote by $\breve\BS$ the set of simple reflections of $\brW_a$ and by $\breve \ell$ the length function on $\brW$.

The following relation between the Iwahori-Weyl groups over $F$ and over $\brF$ is proved in \cite[Lemma 1.6, Corollary 1.7 \& Proposition 1.11]{Ri}.

\begin{proposition}\label{length-add} We have that $W=\brW^\s$, $W_a =\brW_a^\s$ and $\Omega = \breve \Omega^\s$. Moreover, for $x, y \in W$, if $\ell(xy)=\ell(x)+\ell(y)$, then $\breve \ell(xy)=\breve \ell(x)+\breve \ell(y)$.
\end{proposition}

\section{Straight elements and straight conjugacy classes} \label{straight}

\subsection{Definition} In this section, we work with an arbitrary extended affine Weyl group $W$. Let $\ell$ be the length function on $W$ and $\BS$ be the set of simple reflections. We may write $W$ as $$W= Y \rtimes W_0=\{t^\l w; \l \in Y, w \in W_0\},$$ where $W_0$ is a finite Weyl group and $Y$ is a finitely generated abelian group containing the coroot lattice $\BZ\Phi^\vee$ for $W_0$. Moreover, we have $W=W_a \rtimes \Omega$, where $W_a=\BZ\Phi^\vee \rtimes W_0$ and $\Omega \cong Y / \BZ\Phi^\vee$ is the set of length zero elements. Let $\BS_0=\BS \cap W_0$ be the set of simple reflections in $W_0$.

Let $\th$ be a length-preserving automorphism of $W$ of finite order. We define the $\th$-twisted conjugation action on $W$ by $w \cdot_\th w'=w w' \th(w) \i$ for $w, w' \in W$. We denote by $W/_\th \, W$ the set of $\th$-twisted conjugacy classes of $W$.

An element $x$ of $W$ is called {\it $\th$-straight} if $\ell(x \th(x) \cdots \th^{n-1}(x)) = n \ell(x)$ for $n \in \BZ_{\ge 0}$. A $\th$-twisted conjugacy class of $W$ is called {\it straight} if it contains some $\th$-straight elements. We denote by $W \sslash_\th  W$ the set of straight $\th$-twisted conjugacy classes of $W$.

\subsection{Two arithmetic invariants}
Let $\k: W \to W/W_a \cong \Omega$ be the natural projection. Let $\Omega_{\th} = \Omega / (1-\th)\Omega$ be the set of $\th$-coinvariants. Let $$\k_{\th}: W \to \Omega \to \Omega_{\th}$$ be the composition of $\k$ with the projection map $\Omega \to \Omega_{\th}$. We call this map $\k_\th$ the {\it Kottwitz map}.

Let $Y_\BQ=Y \otimes_{\BZ} \BQ$ and $Y_\BQ^+$ be the set of dominant elements in $Y_\BQ$. Let $w \in W$. Since $\th$ is of finite order and $W_0$ is a finite group, there exists a positive integer $n$ such that $\th^n=1$ and $w \th(w) \cdots \th^{n-1}(w)=t^\l$. We set $\nu_{w, \th}=\l/n \in Y_\BQ$. It is easy to see that $\nu_{w, \th}$ is independent of the choice of $n$. Moreover, we denote by $\bar \nu_{w, \th} \in Y_\BQ^+$ the unique dominant element in the $W_0$-orbit of $\nu_{w, \th}$. Then the map $$W \to Y_\BQ^+, \qquad w \mapsto \bar \nu_{w, \th}$$ is constant on each $\th$-twisted conjugacy class of $W$. We call this map the {\it Newton map}.

We define a map $$\pi_{\th}: W \to \Omega_{\th} \times Y_\BQ^+, \qquad w \mapsto (\k_{\th}(w), \bar \nu_{w, \th}).$$ Then $\pi_{\th}$ factors through a map $W/_\th \, W \to \Omega_\th \times Y_\BQ$ which we still denote by $\pi_{\th}$.

Note that the map $\pi_{\th}: W/_\th \, W \to \Omega_{\th} \times Y_\BQ^+$ is not injective. For example, if $\th$ acts trivially on $W$, then all the conjugacy classes that intersects $W_0$ maps to $(1, 0) \in \Omega \times Y_\BQ^+$. However, this map has a nice section.

\begin{theorem}\cite[Theorem 3.3]{HN}\label{gamma}
The map $\pi_{\th}$ induces a bijection from $W \sslash_{\th} W$ to $\Im(\pi_{\th})$.
\end{theorem}

\subsection{} Now we study $W/_\th \, W_a$, the set of $W_a$-orbits on $W$ for the $\th$-twisted conjugation action.

A $W_a$-orbit on $W$ is called {\it $\th$-straight} if it contains a $\th$-straight element. We denote by $W \sslash_{\th} W_a$ the set of straight $W_a$-orbits on $W$. Then we have the following Cartesian diagram
\[
\xymatrix{
W \sslash_{\th} W_a \ar@{^{(}->}[r] \ar@{->>}[d] & W/_\th \, W_a \ar@{->>}[d] \\
W \sslash_{\th} W \ar@{^{(}->}[r] & W/_\th \, W.
}
\]

Define a map $$\pi_\th^\flat: W \to \Omega \times Y_\BQ^+, \qquad w \mapsto (\k(w), \bar \nu_{w, \th}).$$ This map factors through a map $W/_\th \, W_a \to \Omega \times Y_\BQ^+$ which we still denote by $\pi_\th^\flat$.

We have

\begin{theorem} \label{partial}
The map $\pi_\th^\flat$ induces a bijection from  $W \sslash_{\th} W_a$ to $\Im(\pi_\th^\flat)$.
\end{theorem}

The proof is similar to the proof of \cite[Theorem 3.3]{HN}. We provide a proof here for completeness.

\subsection{Minimal length elements} Before we come to the proof of Theorem \ref{partial}, we recall some results on the minimal length elements of $\th$-twisted conjugacy classes.

Let $w, w' \in W$. We write $w \overset s \to_\th w'$ for some $s \in \BS$ if $w'=s w \th(s)$ and $\ell(w') \le \ell(w)$. Write $w \to_\th w'$ if there exits $s_1, \dots, s_r$ such that $w \overset {s_1} \to_\th \cdots \overset {s_r} \to_\th w'$. Write $w \approx_\th w'$ if $w \to_\th w'$ and $w' \to_\th w$.

We denote by $W_{\th-\min}$ the set of elements in $W$ that are of minimal length in their $\th$-twisted conjugacy classes of $W$.

The following result is proved in \cite[Theorem A \& Proposition 2.7]{HN}.

\begin{theorem}\label{min}
(1) Let $w \in W$. Then there exist a subset $J \subseteq \BS$ with $W_J$ finite, a $\th$-straight element $x \in {}^J W {}^{\th(J)}$, and an element $u \in W_J$ such that $x \th(J) x\i = J$, $u x \in W_{\th-\min}$ and $w \to_\th u x$.

(2) Let $w, w' \in W$ be two $\th$-straight elements which are $\th$-twisted conjugate under $W_a$. Then $w \approx_\th w'$.
\end{theorem}

\subsection{Proof of Theorem \ref{partial}}

Let $w \in W$ and let $J, x, u$ be as in Theorem \ref{min} (1) such that $w \to_\th u x$. As $x \th(J) x\i = J$ and $W_J$ is finite, we have $\nu_{x, \th} = \nu_{ux, \th}$. Therefore, $\pi_\th^\flat(w) = \pi_\th^\flat (u x) = \pi_\th^\flat (x)$ and hence $\pi_\th^\flat (W \sslash_\th W_a) = \Im(\pi_\th^\flat)$.

Let $x, x' \in W$ be two elements, which are $\th$-conjugate by $W_a$ to some $\th$-straight elements such that $\pi_\th^\flat (x) = \pi_\th^\flat (x')$. We show that $x, x'$ are $\th$-conjugate by $W_a$. First note that $\k(x) = \k(x')$. Let $\BS_0 \subset \BS$ be the set of simple reflections of $W_0$. Notice that $\th =\Ad(\o) \circ \th'$ for some $\o \in \Omega$ and some automorphism $\th'$ of $W$ with $\th'(\BS_0)=\BS_0$. Thus, by replacing triple $(x, x', \th)$ with $(x\o, x'\o, \th')$, we can assume further that $\th(\BS_0) = \BS_0$. In particular, $\th$ preserves the set of simple coroots. Write $x = t^\l w$ and $x' \in t^{\l'} w'$ for some $\l, \l' \in Y$ and $w, w' \in W_0$. Up to $\th$-conjugation of $W_0$, we may assume that $\nu_{x, \th} = \nu_{x', \th} =: v$ is dominant.

Let $K = \{s \in \BS_0; s(v)=v\}$ and $L = \BS_0 - K$. We denote by $\Phi_K$ the root system of $K$ and by $\Phi_L$ the root system of $L$. As $w(\th(v)) = v = w'(\th(v))$, we deduce that $v = \th(v) = w(v) = w'(v)$. In particular, $w, w' \in W_K$ and $\th$ preserves the root systems $\Phi_K$ and $\Phi_L$. Moreover, as $\k(x) = \k(x')$, we have $\l' - \l = a + b$ for some $a \in \BZ \Phi_K^\vee$ and $b \in \BZ \Phi_L^\vee$. Choose $n \in \BZ_{\ge 1}$ such that $\th^n = (w\th)^n = (w' \th)^n = 1$. Then $$v = \nu_{x, \th} = \frac{1}{n} \sum_{i=1}^n (w \th)^i(\l)$$ and similarly, $$v = \nu_{x', \th} = \frac{1}{n} \sum_{i=1}^n (w \th)^i(\l') \in \frac{1}{n} \sum_{i=1}^n (w \th)^i(\l) + \frac{1}{n} \sum_{i=1}^n \th^i(b) + \BQ \Phi_K^\vee,$$ which means $\sum_{i=1}^n \th^i(b) = 0$ and hence $b = \l_1 - \th(\l_1)$ for some $\l_1 \in \BZ \Phi_L^\vee$. By replacing $x$ by $t^{\l_1} x t^{-\th(\l_1)}$, we can assume that $\l' - \l \in \BZ\Phi_K^\vee$. Thus, by \cite[Proposition 3.2]{HN}, $x, x'$ are $\th$-conjugate to the same length zero element of $Y \rtimes W_K$ under $\BZ \Phi_K^\vee \rtimes W_K$. So $x, x'$ are $\th$-conjugate by $W_a$ as desired.

\begin{corollary} \label{bij}
If $\th$ acts on $\Omega$ trivially, then the natural projection $W \sslash_\th W_a \to W \sslash_\th W$ is a bijection.
\end{corollary}
\begin{proof}
By definition, the projection $W \sslash_\th W_a \to W \sslash_\th W$ is surjective. Now we prove that the projection is injective. Since $\th$ acts on $\Omega$ trivially, we have $\pi_\th = \pi_\th^\flat$ as maps from $W$ to $\Omega \times Y_\BQ^+$. Now the statement follows from Theorem \ref{gamma} and Theorem \ref{partial}.
\end{proof}

\subsection{}\label{inv} Finally we discuss some variation of the above results, which will be used in the rest part of the paper.

Let $\G \subseteq \Omega$ be a $\th$-stable subgroup. Set $W(\G)=W_a \rtimes \G \subset W$. Let $W/_\th \, W(\G)$ be the set of $W(\G)$-orbits on $W$ for the $\th$-twisted conjugation action and $W  \sslash_{\th} W(\G) \subset W /_\th \, W(\G)$ be the set of straight $W(\G)$-orbits on $W$.

We set $\Omega_{\th, \G} = \Omega / (1-\th)\G$. Let $$\k_{\th, \G}: W \to \Omega \to \Omega_{\th, \G}$$ be the composition of $\k$ with the projection map $\Omega \to \Omega_{\th, \G}$.

We define a map $$\pi_{\th, \G}: W \to \Omega_{\th, \G} \times Y_\BQ^+, \qquad w \mapsto (\k_{\th, \G}(w), \bar \nu_{w, \th}).$$ Then $\pi_{\th, \G}$ factors through a map $W/_\th \, W(\G) \to \Omega_\th \times Y_\BQ$ which we still denote by $\pi_{\th, \G}$.

\begin{theorem} \label{Gamma}
	The map $\pi_{\th, \G}$ induces a bijection from  $W \sslash_{\th} W(\G)$ to $\Im(\pi_{\th, \G})$.
\end{theorem}

\begin{remark}
(1)	If $\G=\Omega$, then $W(\G)=W$ and $\pi_{\th, \G}=\pi_\th$. If $\G=\{1\}$, then $W(\G)=W_a$ and $\pi_{\th, \G}=\pi_\th^\flat$. Thus Theorem \ref{Gamma} contains Theorem \ref{gamma} and Theorem \ref{partial} as special cases.

(2) Theorem \ref{Gamma} may be proved in a similar way as Theorem \ref{gamma} and Theorem \ref{partial}. In below, we deduce Theorem \ref{Gamma} directly from Theorem \ref{partial} instead.
\end{remark}

\begin{proof}
	Consider the following commutative diagrams
	\[
	\xymatrix{
		& W \sslash_{\th} W_a \ar@{_{(}->}[ld] \ar@{^{(}->}[rd] \ar@{->>}[dd] & \\
		W/_{\th} \, W_a  \ar[rr]^-{\pi^\flat_{\th}} \ar@{->>}[dd] & & \Omega \times Y_{\BQ}^+ \ar@{->>}[dd]^-{p_{\th, \G}} \\
		& W \sslash_{\th} W(\G) \ar@{_{(}->}[ld] \ar[rd] \\
		W/_{\th} W(\G) \ar[rr]^-{\pi_{\th, \G}}  & & \Omega_{\th, \G} \times Y_{\BQ}^+.
	}
	\]
	
	We have that $$\Im(\pi_{\th, \G})=\Im(p_{\th, \G} \circ \pi^\flat_{\th})=\Im(p_{\th, \G} \circ \pi^\flat_{\th} \mid_{W\sslash_{\th} W_a})=\Im(\pi_{\th, \G} \mid_{W\sslash_{\th} W(\G)}).$$
	
	On the other hand, suppose that $\nu, \nu' \in W \sslash_{\th} W(\G)$ have the same image under $\pi_{\th, \G}$. Let $\nu_1, \nu_1' \in W \sslash_\th W_a$ be two $\th$-straight classes contained in $\nu, \nu'$ respectively. As $\k_{\th, \G}(\nu) = \k_{\th, \G}(\nu')$, we have $\k(\nu_1)-\k(\nu_1') = -\o + \th(\o)$ for some $\o \in \Omega$. So by replacing $\nu_1'$ by $\o\i \nu_1' \th(\o)$, we can assume further that $\k(\nu_1) = \k(\nu_1')$ and hence $\pi_\th^\flat(\nu_1) = \pi_\th^\flat(\nu_1')$. By Theorem \ref{partial}, we have $\nu_1 = \nu_1'$ and hence $\nu = \nu'$ as desired.
\end{proof}

\section{The set $B(\bG)$ and some variations}\label{decomp}

\subsection{The set $B(\bG)$} Recall that $\s$ is the Frobenius morphism on $\brG$. We define the $\s$-twisted conjugation action on $\brG$ by $g \cdot_\s g'=g g' \s(g) \i$. Let $B(\bG)$ be the set of $\s$-twisted conjugacy classes of $\brG$. The $\s$-twisted conjugacy classes of $\brG$ are classified by Kottwitz in \cite{kottwitz-isoI} and \cite{kottwitz-isoII}.

\begin{theorem}
There is an embedding $$f: B(\bG) \to \breve\Omega_\s \times X_*(\bT)_\BQ^+, \qquad g \mapsto (\breve \k(g), \breve{\bar \nu}(g)).$$	
\end{theorem}
For $g \in \brG$, $\breve \k(g)$ and $\breve{\bar \nu}(g)$ are called the Kottwitz point and the Newton point of $g$ respectively. We skip the definition of the maps $\breve \k$ and $\breve{\bar \nu}$ here. But we would like to mention a different point of view, which is useful in our paper.

The natural embedding $\bN_{\bS}(\brF) \to \brG$ induces the map $$\br\pi_\s: \brW \to B(\bG) \to \breve\Omega_\s \times X_*(\bT)_\BQ^+.$$ By \cite{GHKR} and \cite{He14}, the map $\brW \to B(\bG)$ is surjective. In particular, the image of $B(\bG)$ in $\breve\Omega_\s \times X_*(\bT)_\BQ^+$ equals to the image of $\br\pi_\s$. By Theorem \ref{gamma}, we have a natural bijection $$\brW\sslash_{\s} \brW \cong B(\bG).$$

\subsection{Some decompositions on $\brG$}
Let $\th$ be either an automorphism or a Frobenius morphism of $\brG$. We assume furthermore that $\th(\brI)=\brI$. We denote the induced action on $\brW$ still by $\th$.

%For any $\nu \in \brW \sslash_\th \brW$, we define $$\brG(\nu) = \brG \cdot_\th \brI w \brI,$$ where $w$ is any $\th$-straight element in $\nu$.

Let $\br\G \subseteq \br\Omega$ be a $\th$-stable subgroup. Let $\brG(\G)=\brG \times_{\breve \Omega} \br\G$ be the Cartesian product.
%Consider the following Cartesian product
%\[
%\xymatrix{
%\brG(\G) \ar@{->}[r] \ar@{->}[d] & \br\G \ar@{->}[d] \\
%\brG \ar@{->}[r] & \br\Omega.
%}
%\]
Notice that the Iwahori-Weyl group of $\bG_{\br F}$ is $\brW(\br\G)=\brW_a \rtimes \br\G$. For any $\breve \nu \in \brW\sslash_{\th} \brW(\G)$, we define $$\brG(\br\G; \breve \nu)= \brG(\br\G) \cdot_\th \brI w \brI,$$ where $w$ is any $\th$-straight element in $\breve \nu$. Note that if $w'$ be another $\th$-straight in $\breve \nu$, then by Theorem \ref{min} (2) $w \approx_\th \t w' \th(\t) \i$ for some $\t \in \br\G$. Applying the proof of \cite[Lemma 3.1 (2)]{He14}, we deduce that $\brG(\br\G) \cdot_\th \brI w \brI = \brG(\br\G) \cdot_\th \brI w' \brI$. Therefore the definition of $\brG(\br\G; \breve \nu)$ is independent of the choice of $w$. If $\br\G$ is trivial, $\brG(\br\G) = \brG_1$ is the subgroup generated by all the parahoric subgroups. In this case, we write $\brG(\breve \nu)^\flat$ for $\brG(\br\G; \breve \nu)$.

Now we prove the main result of this section.

\begin{theorem} \label{dec} We have that
$$\brG = \sqcup_{\breve \nu \in \breve W \sslash_\th \brW(\br\G)} \brG(\br\G; \breve \nu).$$

Moreover, if $\th$ is a Frobenius morphism, then the $\th$-twisted conjugation action of $\brG(\br\G)$ on $\brG(\br\G; \breve \nu)$ is transitive.
\end{theorem}
\begin{proof}
The statement follows similarly as \cite[Theorem 3.7]{He14}. We sketch a proof for completeness.

Combining Theorem \ref{min} (2) with the reduction method in \cite[Lemma 3.1]{He14}, we have $$\brG = \cup_{x \in \brW_{\th-\min}} \brG_1 \cdot_\th \brI x \brI.$$ Let $w \in \brW_{\th-\min}$ and let $J, x, u$ be as in Theorem \ref{min} (2) such that $w \to_\th u x$ and hence $w \approx_\th u x$. Thus, $\brG_1 \cdot_\th \brI w \brI = \brG_1 \cdot_\th \brI u x \brI$ by \cite[Lemma 3.1 (2)]{He14}. As in the proof of \cite[Lemma 3.2]{He14}, we let $\sP \subseteq \brG$ be the parahoric subgroup corresponding to $J$ and let $\sU$ be its pro-unipotent radical. Denote by $\bar \sP = \sP / \sU$ the reductive group and $p \mapsto \bar p$ the projection from $\sP$ to $\bar \sP$. As $x \in {}^J W {}^{\th(J)}$ and $x \th(J) x\i = J$, the map $\bar p \mapsto \dot x \th(\bar p) \dot x\i$ gives an automorphism or a Frobenius morphism $\iota$ of $\bar \sP$, which preserves the Borel subgroup $\bar \sB = (\brI \cap \sP) / \sU$ of $\bar \sP$. Applying Lang-Steinberg theorem, we have $\bar \sP = \bar \sP \cdot_\iota \bar \sB$, which means $$\sP \cdot_\th \brI u x \brI \subseteq \sP \cdot_\th (\sP x) \subseteq  \sP \cdot_\th \brI x \brI.$$ So $\brG_1 \cdot_\th \brI w \brI = \brG_1 \cdot_\th \brI u x \brI = \brG_1 \cdot_\th \brI x \brI \subseteq \brG(\br\G) \cdot_\th \brI x \brI$. Hence $$\brG = \cup_{\breve \nu \in \brW \sslash_\th \brW(\br\G)} \brG(\br\G; \breve \nu).$$

Let $\breve \nu, \breve \nu' \in \brW \sslash_\th \brW(\br\G)$. Suppose that $\brG(\br\G; \breve \nu) =\brG(\br\G; \breve \nu')$. By the proof of \cite[Proposition 3.6]{He14}, $\br\pi_{\th, \br\G} (\breve \nu) = \br\pi_{\th, \br\G} (\breve \nu')$, where $\br\pi_{\th, \br\G}$ is the map on $\brW \sslash_{\th}\brW(\br\G)$ defined in \S\ref{inv}. By Theorem \ref{Gamma}, we have $\breve \nu =\breve \nu'$ and hence $$\brG=\sqcup_{\breve \nu \in \breve W \sslash_\th \brW(\br\G)} \brG(\br\G; \breve \nu).$$

For the ``moreover part, by the above arguments it suffices to show that the $\th$-twisted conjugation action of $\br I \subseteq \brG(\br\G)$ on $\brI x \brI$ is transitive if $x$ is $\th$-straight. The statement then follows from \cite[Proposition 6.3]{GHKR} and \cite[Proposition 3.4.3]{GHN}.
\end{proof}

\begin{corollary}
	Let $\breve \nu \in \brW\sslash_{\th} \brW(\br\G)$. Then we have
	\[\brG(\br\G; \breve \nu)=\sqcup_{\breve \nu' \in \brW\sslash_{\th} \brW_a; \breve \nu' \subseteq \breve \nu} \brG(\breve \nu')^\flat.\]
\end{corollary}

\begin{proof}
	By definition, for any $\breve \nu' \in \brW\sslash_{\th} \brW_a$ with $\breve \nu' \subseteq \breve \nu$, we have $\brG(\breve \nu')^\flat \subset \brG(\breve \nu)_\G$.
	
	On the other hand, as $\brG(\br\G) = \sqcup_{\o \in \br\G} \brG_1 \o$,
	we have $$\brG(\br\G; \breve \nu) = \brG(\br\G) \cdot_\th I x I = \cup_{\o \in \br\G} \brG_1 \cdot_\th \brI \o\i x \th(\o) \brI \subseteq \sqcup_{\breve \nu' \in \brW\sslash_{\th} \brW_a; \breve \nu' \subseteq \breve \nu} \brG(\breve \nu')^\flat,$$ where $x$ is any $\th$-straight representative of $\breve \nu$.
\end{proof}

\subsection{The case $\th=\s$}\label{th-s} In this subsection, we consider the case that $\th = \s$ is the Frobenius morphism. If $\br\G = \br\Omega$, then Theorem \ref{dec} gives a natural bijection $B(\bG) \cong \brW \sslash_{\s} \brW$.

Now we denote by $B(\bG)^\flat$ the set of $\brG_1$-orbits on $\brG$ for the $\s$-twisted conjugation action. The set $B(\bG)^\flat$ occurs in the study of connected components of Shimura varieties \cite{Ha}.

Theorem \ref{dec} gives a natural bijection $B(\bG)^\flat \cong \brW \sslash_{\s} \brW_a$ and we have the following commutative diagram
\[
\xymatrix{
B(\bG)^\flat \ar[r]^-\cong \ar@{->>}[d] & \brW \sslash_{\s} \brW_a \ar@{->>}[d] \\
B(\bG) \ar[r]^-\cong & \brW \sslash_{\s} \brW.
}
\]

If $\bG$ is residually split, then the action of $\s$ on $\br\Omega$ is trivial. By Corollary \ref{bij}, the natural projection map the natural projection $W \sslash_\s W_a \to W \sslash_\s W$ is bijective. Hence the natural projection map $B(\bG)^\flat \to B(\bG)$ is a bijection.

\begin{example}\label{ex-th-s}
	If $\bG$ is not residually split, then  in general, $\brW \sslash_\th \brW \neq \brW \sslash_\th \brW_a$ and $B(\bG) \neq B(\bG)^\flat$. For example, let $\bG=\Res_{E/F} \SL_2$, where $E / F$ is an unramified extension of degree two. In this case, $\brW = \brW' \times \brW'$ and $\s$ permutes the two copies of the Iwahori-Weyl group $\brW'$ of $\SL_2$. Let $\o \in \brW'$ be the unique length zero element of order 2. Then the two length zero elements $(1, 1), (\o, \o) \in \brW$ are in the same $\s$-conjugacy class, but not in the same $\brW_a$-$\s$-conjugacy class.
\end{example}

\subsection{The case that $\theta$ is a group automorphism}
In \cite[\S 2]{KV}, Kottwitz and Viehmann defined the Newton homomorphism $\nu_\th(g): \bold D \to \bold G$ and the Newton point $\bar \nu_\th(g) \in Y_\BQ^+$ for each $g \in \brG$, where $\bold D$ is the diagonalizable group scheme over $F$ with character group $\BQ$.

We show that their definition of Newton points coincides with ours. Indeed, as $\nu_\th(h\i g \th(h)) = \Int(h\i) \circ \nu_\th(g)$ and that $\brG = \sqcup_{\br\nu \in \brW \sslash_\th \brW} \brG(\nu)$, it suffices to show that $\nu_\th(g) = \nu_{x, \th}$ for $g \in \br I x \brI$ if $x$ is $\th$-straight. By the proof of \cite[Theorem 2.1.2]{GHKR}, $\br I \cdot_\th \br I_{\bold L} x \br I_{\bold L} = \br I x \br I$, where $\br I_{\bold L} =\bold L(\br F) \cap \br I$ and $\bold L$ is the centralizer of $\nu_{\th, x}$ in $\bold G$. Thus, by replacing $\bold G$ with $\bold L$, we can assume further that $x \in \br \Omega$. Then by the definition in \cite[\S 2]{KV}, we conclude that $\nu_\th(g) = \nu_\th(x) = \nu_{x, \th}$ as desired.

\section{A comparison between $W \sslash W$ and $\brW \sslash_\th \brW$} \label{embed}
%In rest of the paper, we adopt the notation introduced in \S\ref{setup}. Let $\th$ be an automorphism of $\brG$ such that $\th(\brI)=\brI$ and $\th \circ \s = \s \circ \th$. We also denote by $\th$ the induced automorphism of the Iwahori-Weyl group $\brW$. Let $\G \subset \Omega$ be a $\th$-stable subgroup. Recall that $W(\G)= W_a \rtimes \G \subseteq W$ and $\brW(\G)=\brW_a \rtimes \G \subseteq \brW$.

\begin{theorem} \label{inject} Suppose that the induced actions of $\th$ and $\s$ on $\brW$ commute. Then the natural embedding $W \hookrightarrow \brW$ induces an injection $$i: W \sslash_\th W(\G)  \hookrightarrow \brW \sslash_\th \brW(\G).$$
\end{theorem}

\subsection{Some discussions}
	Before we give the proof, we would like to mention the cases where such injection can been seen easily and point out the subtlety in the general case.
	
	We have $\brW=\br Y \rtimes \brW_0$ and $W=Y \rtimes W_0$. If $\bG$ is residually split, then $\s$ acts trivially on $\brW$ and hence we have $\brW=W$. In this case, $W\sslash_{\th} W=\brW \sslash_{\th} \brW$.
	
	If $\bG$ is quasi-split, then we have $\Omega=\br \Omega^\s$, $W_0=\brW_0^\s$, and $Y=\br Y^\s$. Then we have a natural embedding $\Omega_\th \to \br \Omega_{\th, \Omega}$ and $Y_{\BQ}^+ \subset \br Y_{\BQ}^+$. By Theorem \ref{Gamma}, we have a natural embedding $W \sslash_\th W  \hookrightarrow \brW \sslash_\th \brW(\Omega)$.
	
	The cases above are easy to handle as the Newton maps over $F$ and over $\brF$ are compatible. In general, one may view $Y_\BQ$ as an affine subspace (but not necessarily a linear subspace) of $\br Y_\BQ$, and a $W_0$-dominant point in $Y_\BQ$ may not be a $\br W_0$-dominant point in $\br Y_\BQ$. This causes difficulty to compare the Newton maps over $F$ and over $\brF$.
	
\begin{example} \label{ex1} Here we consider the case where $\brW$ is of type affine $A_2$ with $\br\BS=\{s_0, s_1, s_2\}$ and the induced action of $\s$ on $\br\BS$ permutes $s_0$ and $s_1$, and fixes $s_2$. Let $\{\a_1, \a_2\}$ be the set of simple roots for $\br W_0$. Then
	\begin{gather*} \br Y_\BQ = \BQ \a^\vee \oplus \BQ \b^\vee, \\ Y_{\BQ} = \br Y_{\BQ}^\s = \frac{1}{3}\a_2^\vee + \BQ \a_1^\vee.
	\end{gather*} So $Y_{\BQ}$ is an affine subspace of $\br Y_{\BQ}$ not containing the origin $0$. The dominant chambers of $\br Y_\BR$ and $Y_\BR$ are illustrated in Figure \ref{fig-1}.
\end{example}

\begin{example} \label{ex} Here we consider the case where $\brW$ is of type affine $B_2$ and the induced action of $\s$ on $\brW$ is nontrivial. In this case, we fix a set $\{\a, \b\}$ of simple roots for $\br W_0$, where $\a$ is the long root and $\b$ is the short root. Then
	\begin{gather*} \br Y_\BQ = \BQ \a^\vee \oplus \BQ \b^\vee, \\ Y_{\BQ} = \br Y_{\BQ}^\s = \frac{1}{4}(\a^\vee + \b^\vee) + \BQ \a^\vee.
	\end{gather*} So $Y_{\BQ}$ is an affine subspace of $\br Y_{\BQ}$ not containing the origin $0$. The dominant chambers of $\br Y_\BR$ and $Y_\BR$ are illustrated in Figure \ref{fig-1}.
	
	\begin{figure}
		\begin{tikzpicture}[scale=.8]
		
		% \draw[fill=lightgray] (-2, -2) -- (4, 4) -- (4, -2) -- cycle;
		\draw[fill, gray] (0, 0) -- (2, 3.464) -- (4, 0) -- cycle;
		%\draw (0, 0) -- (1,1.732) -- (2, 0) -- cycle;
		%\draw (0, 0) -- (-1,-1.732) -- (1, -1.732) -- cycle;
		
		%\draw[dotted,color=blue!50,line width=.5mm] (-5, 2) -- (0, 2);
		%\draw(0, 2) -- (5, 2);
		\draw[black,line width=.7mm](1,1.732) -- (5, -0.5774);
		%\draw[dashed] (2, 2) -- (5, 2);
		\draw[dotted] (1,1.732) -- (-3, 4.0414);
		\draw[dash dot] (0,0) -- (6, -3.464);
		\draw[dotted,black,line width=.7mm] (0,0) -- (4, 2.3094);
		\draw[fill, gray] (1,1.732) circle [radius=.17cm];

		\draw (-6,3.464) -- (6,3.464);
		\draw (-6,0) -- (6,0);
		\draw (-6,-3.464) -- (6,-3.464);
		%\draw (2.5,0.866) -- (1,-1.732);
		%\draw (-2.5,0.866) -- (-1,-1.732);
		
		\draw (-1,5.196) -- (-5,-1.732);
		\draw (-3,5.196) -- (3,-5.196);
		\draw (-5,1.732) -- (-1,-5.196);
		\draw (1,5.196) -- (5,-1.732);
		\draw (3,5.196) -- (-3,-5.196);
		\draw (5,1.732) -- (1,-5.196);
		% \draw (4,-5) -- (4,5);
		
		%\draw (-4.7,-4.7) -- (2, 2);
		%\draw [dash dot](2,2) -- (2, 4.7);
		%\draw [dotted,line width=.5mm](2,2) -- (2, 4.7);
		
		\draw[fill, black] (0,0) circle [radius=.17cm];
		%\draw[fill, gray] (0,2) circle [radius=.17cm];
		\end{tikzpicture}
		
		\begin{tikzpicture}[scale=.8]
		
		% \draw[fill=lightgray] (-2, -2) -- (4, 4) -- (4, -2) -- cycle;
		\draw[fill=gray] (0, 4) -- (2, 2) -- (0, 2) -- cycle;
		\draw[fill=gray] (0, 0) -- (0,2) -- (2, 2) -- cycle;
		% \draw[fill=gray] ( 2,  2) -- (4, 4) -- (4, 2) -- cycle;
		
		%\draw[dotted,color=blue!50,line width=.5mm] (-5, 2) -- (0, 2);
		%\draw(0, 2) -- (5, 2);
		\draw[black,line width=.7mm](0, 2) -- (5, 2);
		
		%\draw[dashed] (2, 2) -- (5, 2);
		\draw[dotted] (-5, 2) -- (0, 2);
		
		% \draw (-5,-4) -- (5,-4);
		% \draw (-5,-2) -- (5,-2);
		\draw (-5,0) -- (0,0);
		%\draw (-5,2) -- (5,2);
		% \draw (-5,4) -- (5,4);
		
		% \draw (-4,0) -- (-4,5);
		% \draw (-2,0) -- (-2,5);
		\draw (0,-5) -- (0,0);
		\draw[dash dot] (0, 0) -- (5, 0);
		%\draw [dotted,line width=.5mm](0,2) -- (0,5);
		
		% \draw (2,-5) -- (2,5);
		% \draw (4,-5) -- (4,5);
		
		\draw (-4.7,-4.7) -- (2, 2);
		\draw [dotted,black,line width=.7mm](0, 0) -- (0, 5);
		%\draw [dotted,line width=.5mm](2,2) -- (2, 4.7);
		\draw (2,2) -- (4.7, 4.7);
		\draw (-4.7, 4.7) -- (-2,2);
		\draw (4.7, -4.7)-- (-2,2);
		
		\draw (-2, 2) -- (0.7, 4.7);
		\draw (-4.7, -0.7) -- (-2, 2);
		\draw (-0.7, -4.7) -- (4.7, 0.7);
		%\draw [dotted,line width=.5mm](-0.7, 4.7) -- (2, 2);
		\draw (-0.7, 4.7) -- (4.7, -0.7);
		\draw (-4.7, 0.7) -- (0.7, -4.7);
		% \draw (-4.7, -3.3) -- (-3.3, -4.7);
		% \draw (4.7, 3.3) -- (3.3, 4.7);
		% \draw (-4.7, 3.3) -- (-3.3, 4.7);
		% \draw (4.7, -3.3) -- (3.3, -4.7);
		\draw[fill, black] (0,0) circle [radius=.17cm];
		\draw[fill, gray] (0,2) circle [radius=.17cm];
		\end{tikzpicture}
		
		\caption{This is an illustration for Example \ref{ex1} and Example \ref{ex} respectively. The thick dot is the origin of the apartment $\br Y_\BR$. The alcove in grey is the based alcove. The Frobenius $\s$ acts on $\br Y_\BR$ by the reflection along the affine line containing the thick segment. Hence it is the apartment $Y_\BR = \br Y_\BR^\s$. The grey dot is a special vertex $e$ in $\sA$. The right half of the affine line through the grey dot is the dominant chamber of $Y_\BR$. The corresponding half line (by  minus $e$) is the dashed-dotted half line through the origin of the apartment $\br Y_\BR$. The corresponding dominant elements in $\br Y_\BR$ is the dotted half line through the origin. }\label{fig-1}
	\end{figure}
\end{example}

\subsection{Proof of Theorem \ref{inject}}
Let $x$ be a $\th$-straight element in $W$. Then for any $n \in \BZ_{\ge 0}$, we have $\ell(x \th(x) \cdots \th^{n-1}(x))=n \ell(x)$. By Proposition \ref{length-add}, we have that $\breve \ell(x \th(x) \cdots \th^{n-1}(x))=n \breve \ell(x)$. Thus $x$ is also a $\th$-straight element in $\brW$. Note that $W(\G) \subset \brW(\G)$. Thus the natural embedding map $W \to \brW$ induces a map $i: W \sslash_{\th} W(\G) \to \brW \sslash_{\th} \brW(\G)$.

Next we show that the map is injective. Let $x, y \in W$ be two $\th$-straight elements. In view of Proposition \ref{Gamma}, we need to show that $\pi_{\th, \G}(x) = \pi_{\th, \G}(y)$ if $\br \pi_{\th, \G}(x) = \br \pi_{\th, \G}(y)$, where $\pi_{\th, \G}, \br \pi_{\th, \G}$ are as in \S\ref{inv} defined for $W$ and $\brW$ respectively.

By the definition of $\breve \pi_{\th, \G}$, we have $\k_{\th, \G}(x) = \k_{\th, \G}(y) \in \Omega_{\th, \G}$. Next we compare the Newton points.

We fix a special vertex $e$ of the alcove $\fka_C$ in $\sA$ and identify $\sA$ with an affine subspace in $\breve \sA$ that contains $e$. Let $w \in W$. Let $n$ be sufficiently divisible so that $\th^n=1$ and
\begin{gather*}
w \th(w) \cdots \th^{n-1}(w)=t^{\l} \in W; \\
w \th(w) \cdots \th^{n-1}(w)=t^{\breve \l} \in \brW.
\end{gather*}

We regard $w \th(w) \cdots \th^{n-1}(w)$ as affine transformations on $\sA$ and $\breve \sA$. Then $w \th(w) \cdots \th^{n-1}(w)$ sends $e$ to $e+\l$ and also to $e+\breve \l$. Therefore $\l=\breve \l$. Note that in $\sA$, the special vertex $e$ is considered as the origin. Hence $\nu_{w, \th}$ (when $w$ is regarded as an element in $W$) defined in \S\ref{inv} equals to $e+\l/n$ when considered as an element in $\br Y_\BQ$. On the other hand, $\breve \nu_{w, \th}$ (when $w$ is regarded as an element in $\brW$) equals to $\br \l/n=\l/n \in \br Y_\BQ$.

Up to the action of a suitable element in $W_0$, we may assume that $e+\nu_{x, \th}$ and $e+\nu_{y, \th}$ are in the same Weyl chamber of $\sA$ (with origin $e$). We claim that $\nu_{x, \th} = \nu_{y, \th}$ and this will finish the proof. As $\nu_{x, \th}, \nu_{y, \th}$ are conjugate by $\brW_0$, it suffices to show there is no root hyperplane of $\brW_0$ separating $\nu_{x, \th}$ from $\nu_{y, \th}$. Otherwise, let $\g$ be such a root, and denote by $\bar\g$ the corresponding relative root for $W_0$. Then the root hyperplane $H_{\bar\g} \subseteq \sA$ of $\bar \g$ passing through $e$ separates $e + \nu_{x, \th}$ from $e+\nu_{y, \th}$, contradicting that $e+\nu_{x, \th}, e+\nu_{y, \th}$ are in the same Weyl chamber of $\sA$. The proof is finished.

%Let $w \mapsto \bar \nu_{w, \th}$ be the Newton map for $W$. Then $\bar \nu_{w, \th}$ can viewed as the set $$\{\nu_{u\i w \th(u), \th}; u \in W\} = p(W)(\nu_{w, \th}) = p(W_0)(\nu_{w, \th}),$$ where $p: W \to \brW_0 \subseteq \End(\breve Y_\BQ)$ is the natural projection. So it remains to show that $\nu_{x, \th}, \nu_{y, \th}$ are conjugate by the relative Weyl group $p(W_0)$. Fix a special point $e$ of $\sA$ and identify its stabilizer in $W$ with $W_0$. Then there exists $w \in W_0$ such that $w(e+\nu_{x, \th}) = e+p(w)(\nu_{x, \th})$ and $e+\nu_{y, \th}$ are in the same Weyl chamber of $\sA$ (with origin $e$). We claim that $p(w)(\nu_{x, \th}) = \nu_{y, \th}$ and this will finish the proof. As $\nu_{x, \th}, \nu_{y, \th}$ are conjugate by $\brW_0$, it suffices to show there is no root hyperplane of $\brW_0$ separating $p(w)(\nu_{x, \th})$ from $\nu_{y, \th}$. Otherwise, let $\g$ be such a root, and denote by $\bar\g$ the corresponding relative root for $W_0$. Then the root hyperplane $H_{\bar\g} \subseteq \sA$ of $\bar \g$ passing through $e$ separates $e + p(w)(\nu_{x, \th})$ from $e+\nu_{y, \th}$, contradicting the fact that $e+p(w)(\nu_{x, \th}), e+\nu_{y, \th}$ are in the same Weyl chamber of $\sA$. The proof is finished.

\section{A geometric interpretation of Newton strata}
\subsection{Newton strata of $G$}\label{5.1}
Suppose that the action of $\th$ and $\s$ on $\brG$ commutes. Recall that $W_{\th-\min}$ is the set of elements in $W$ that are of minimal length in their $\th$-twisted conjugacy classes of $W$. %We have the map $\pi_\th: W \to \Omega_{\th} \times Y_{\BQ}^+$ whose image equals to the image of $\pi_\th \mid_{W\sslash_{\th} W}$.

In \cite{Newton}, the first author introduced the Newton strata of $G$. We recall the definition here.
For $\nu \in W \sslash_\th W$, the Newton stratum $G(\nu)$ is defined by $$G(\nu) = \cup_{w \in W_{\th-\min}; \pi_\th(w)=\pi_\th(\nu)} G \cdot_\th I w I.$$

\begin{theorem} \cite[Theorem A]{Newton}\label{newtonG}
	We have the following Newton decomposition $$G = \sqcup_{\nu \in W \sslash_\th W} G(\nu).$$
\end{theorem}

The above definition of Newton strata plays an essential role in the new proof of Howe's conjecture
given in \cite[\S 5]{Newton}. However, the definition is rather technical. Now we give a geometric interpretation of the Newton strata of $G$, using the decomposition of $\brG$ we discussed in \S \ref{decomp}.

\subsection{Geometric interpretation}\label{5.2} Set $\brW^\natural = \brW(\Omega)$ and $\brG^\natural = \brG(\Omega)$. Let $\br \pi_\th^\natural = \pi_{\th, \Omega}$ be as in \S\ref{inv} defined for $\brW$, viewing $\Omega$ as a subgroup of $\br\Omega$. For $\breve \nu \in \brW \sslash_\th \brW^\natural$, set $\brG(\breve \nu)^\natural = \brG(\Omega; \breve \nu)$.

\begin{theorem}\label{main-tt}
Let $\nu \in W \sslash_\th W$ and $i(\nu)$ be the corresponding element in $\brW \sslash_{\th} \brW^\natural$. Then we have that $$G(\nu)=G \cap \breve G(i(\nu))^\natural.$$
\end{theorem}

\begin{proof}
	We show that \[\tag{a} G(\nu) \subset \breve G(i(\nu))^\natural.\]
	
	Let $w \in W_{\th-\min}$ with $\pi_\th(w)=\pi_\th(\nu)$. Let $J, x, u$ be as in Theorem \ref{min} (1) such that $w \to_\th u x$. Then $\pi_\th(w) = \pi_\th(u x) = \pi_\th(x) =\pi_\th(\nu)$. By Theorem \ref{gamma}, $x$ is a $\th$-straight element in $\nu$. Then $x$ is also a $\th$-straight element in $i(\nu)$.
	
	By \cite[Lemma 3.1 \& Lemma 3.2]{He14}, we have $$I w I \subseteq \brG_1 \cdot_\th (\brI w \brI) = \brG_1 \cdot_\th (\brI x u \brI) = \brG_1 \cdot_\th \brI x \brI.$$ Thus $G \cdot_\th I w I \subset \brG^\natural \cdot_\th (\brG_1 \cdot_\th \brI x \br I)=\brG(i(\nu))^\natural$. (a) is proved.
	
	By Theorem \ref{inject}, the map $i$ is injective. Hence by Theorem \ref{dec}, the union $\cup_{\nu \in W \sslash_{\th} W}  \brG(i(\nu))^\natural$ is a disjoint union.
	We have \[\tag{b}G \subseteq \sqcup_{\nu' \in W \sslash_{\th} W} G(\nu') \subseteq \sqcup_{\nu' \in W \sslash_{\th} W} (G \cap \brG(i(\nu'))^\natural) \subseteq G \cap \bigl(\sqcup_{\nu' \in W \sslash_{\th} W}  \brG(i(\nu'))^\natural\bigr) \subseteq G.\] Here the first inclusion follows from Theorem \ref{newtonG} and the second inclusion follows from (a) above.
	
	Note that in (b), all the inclusions must be equalities. In particular, $G(\nu')=G \cap \brG(i(\nu'))^\natural$ for all $\nu' \in W\sslash_{\th} W$.
\end{proof}

\subsection{Newton strata of $G$ and $B(\bG)^\flat$} We consider the case where $\th=\s$. In this case, $\th$ acts trivially on $G$ and the $\th$-twisted conjugation action on $G$ is the usual conjugation action. We have the Newton decomposition $$G=\sqcup_{\nu \in W \sslash W} G(\nu).$$

By Corollary \ref{bij}, the set $W\sslash W$ is naturally bijective to $W\sslash W_a$. By Theorem \ref{inject}, we have a natural injection $W\sslash W_a \lto \brW \sslash_\th \brW_a$. Similar to Theorem \ref{main-tt}, we have that

\begin{theorem}\label{5-5}
	Let $\th=\s$. Then for any $\nu \in W\sslash W_a \cong W\sslash W$, we have $$G(\nu)=G \cap [\breve \nu]^\flat,$$ where $[\breve \nu]^\flat \in B(\bG)^\flat$ is the $\brG_1$-orbit on $\brG$ for the $\s$-twisted conjugation action.
\end{theorem}

Note that as we discussed in \S \ref{th-s} and Example \ref{ex-th-s}, the statement is not true if we replace $B(\bG)^\flat$ by $B(\bG)$.

\end{document}